\documentclass{article}
\usepackage{amsmath,amssymb,amsthm}
\usepackage{hyperref}
\usepackage{tikz-cd}
\usepackage{quiver}
\usepackage{graphicx}
\theoremstyle{plain}
\newtheorem{theorem}{Theorem}[section]
\newtheorem{lemma}[theorem]{Lemma}
\newtheorem{proposition}[theorem]{Proposition}
\newtheorem{corollary}[theorem]{Corollary}
\newtheorem{conjecture}[theorem]{Conjecture}

\theoremstyle{definition}

\theoremstyle{remark}
\newtheorem{remark}[theorem]{Remark}

\begin{document}

\title{Rational Homology Ribbon Cobordism is a Partial Order}
\author{William Ghanem}
\date{\today}
\maketitle

\begin{abstract}
We prove that rational homology ribbon cobordism defines a partial order on the set of homeomorphism classes of closed, connected, oriented 3-manifolds. More precisely, if $W_0 \leq W_1$ and $W_1 \leq W_0$, then $W_0$ and $W_1$ are homeomorphic by an orientation preserving homeomorphism.
\end{abstract}
\section{Introduction}
In \cite{Gordon1981RibbonConcordance}, Gordon introduced ribbon concordance of knots and conjectured that ribbon concordance defines a partial order on knots in $S^3$. In \cite{Agol2022RibbonConcordancePartialOrdering}, Agol proved this conjecture and suggested that the proof might resolve the analogous conjecture for rational homology ribbon cobordism, originally posed as Conjecture 1.1 in \cite{DaemiLidmanVelaVickWong2022RibbonHomologyCobordisms}. In \cite{Huber2023RibbonCobordismsPartialOrder}, Huber proves this conjecture in the case of aspherical 3-manifolds. Further work by Huber focuses on lens spaces: in \cite{Huber2021RibbonCobordismsLensSpaces}, Huber shows when there exists a rational homology ribbon cobordism between connected sums of lens spaces. Using results for lens spaces, Friedl-Misev-Zentner subsequently prove the irreducible case \cite{FriedlMisevZentner2025RibbonRH}. In this paper, we extend this result to all closed, connected, oriented 3-manifolds, thereby proving Conjecture 1.1 of \cite{DaemiLidmanVelaVickWong2022RibbonHomologyCobordisms}. Below, we recall the relevant definitions and notation needed to state Conjecture 1.1. \\

Let $W_0$, $W_1$ be closed, connected, oriented 3-manifolds. A \textit{ribbon cobordism} from $W_0$ to $W_1$ is a cobordism $W$ with $\partial W=-W_0\sqcup W_1$ such that $W$ is constructed from $W_0 \times [0,1]$ by attaching 1-handles and 2-handles to $W_0 \times \{1\}$. A \textit{rational homology ribbon cobordism} from $W_0$ to $W_1$ is a ribbon cobordism $W$ from $W_0$ to $W_1$ such that the inclusions $W_0 \hookrightarrow W$ and $W_1 \hookrightarrow W$ induce isomorphisms on rational homology groups. If there exists a rational homology ribbon cobordism between $W_0$ and $W_1$, we write $W_0 \leq W_1$. The conjecture we consider is the following, as stated in \cite{DaemiLidmanVelaVickWong2022RibbonHomologyCobordisms}.
\begin{conjecture}
The preorder on the set of homeomorphism classes of closed, connected, oriented 3-manifolds given by ribbon rational homology cobordism is a partial order, i.e. if one has $W_0\leq W_1$ and $W_1 \leq W_0$, then $W_0$ and $W_1$ are homeomorphic.
\end{conjecture}
In our work, we prove the conjecture. Reflexivity is given by the interval $W_0 \times [0,1]$. For transitivity, suppose that $W_0 \leq W_1$ via $W$ and $W_1 \leq W_2$ via $V$. Then $W_0\leq W_2$ via $WV$, the rational homology ribbon cobordism given by attaching $W$ to $V$ along $W_1$. In particular, all that remains is antisymmetry; we prove the following theorem.
\begin{theorem}
Suppose that $W_0 \leq W_1$ and $W_1 \leq W_0$. Then $W_0$ and $W_1$ are oriented homeomorphic.
\end{theorem}
Finally, we introduce conventions for rational homology ribbon cobordisms. Let $W$ be a rational homology ribbon cobordism from $W_0$ to $W_1$. We denote the inclusion of $W_i$ by $\iota_i$, $i=0,1$. By Proposition 2.1 of \cite{DaemiLidmanVelaVickWong2022RibbonHomologyCobordisms}, $\iota_0$ induces an injection on $\pi_1$ and $\iota_1$ induces a surjection. If $(\iota_1)_*$ is also injective on $\pi_1$, and therefore an isomorphism, we denote by $\alpha$ the map on fundamental groups defined by the composition $\alpha=(\iota_1)_*^{-1}\circ(\iota_0)_*$. In particular, if we also have that $W_1\leq W_0$, then by Corollary 1.3 of \cite{FriedlMisevZentner2025RibbonRH}, $(\iota_1)_*$ is necessarily an isomorphism and the map $\alpha$ is defined.\\

We divide the proof into two parts: in the first part, we demonstrate how it suffices to prove the conjecture in the case that $W_0$ and $W_1$ have no $S^1 \times S^2$ summands. In the second part, we provide a proof for this subcase, which, in turn, proves Theorem 1.2.

\section{Preliminaries}
We begin with some definitions and conventions. The \textit{connected sum} of two closed, connected, oriented 3-manifolds $M$ and $N$ is given by $M\#N=(M \backslash B_M^{\circ} \cup N\backslash B^{\circ}_N)/(\partial B_M=\partial B_N)$ where $B_M$ and $B_N$ are 3-balls in $M$ and $N$ and the quotient refers to an orientation reversing homeomorphism of their bounding spheres. We call an \textit{identification} of $Y$ with $M \# N$ an orientation preserving homeomorphism $X\rightarrow Y$ where $X$ is a 3-manifold obtained by performing the connected sum of $M$ and $N$ with some 3-balls and some 2-sphere identification. An identification $Y\cong M\#N$ comes with an isomorphism $\pi_1(M)*\pi_1(N)\rightarrow \pi_1(Y)$ defined on $\pi_1(M)$ by the composition $\pi_1(M)\xleftarrow[]{\cong}\pi_1(M \backslash B_M)\rightarrow \pi_1(X)\rightarrow \pi_1(Y)$ where all arrows are induced by inclusion. Note that we have fixed a basepoint in $\partial B_M$. The definition on $\pi_1(N)$ is defined identically; we use the corresponding basepoint in $X$ from $\partial B_M$. We typically just write $Y\cong M\#N$ to declare some identification with $M\#N$ and refer to maps $\pi_1(M)\xrightarrow{i_{\pi_1(M)}}\pi_1(Y)$ and $\pi_1(N)\xrightarrow{i_{\pi_1(N)}} \pi_1(Y)$ as defined above. Further, we often write $\pi_1(Y)=\pi_1(M) *\pi_1(N)$ where each factor refers to the subgroup provided by the above inclusions. Larger connected sums $Y\cong M_1 \#...\# M_n$ are defined similarly and come equipped with similarly defined maps on $\pi_1$. Motivated from the connection between connected sums and free products, we prove the following short lemma.
\begin{lemma}
    Let $G=G_1 *...*G_n*F_m$ and $H=H_1*...*H_k*F_l$ be Grushko decompositions for finitely generated groups $G$ and $H$. Suppose that $f:G\rightarrow H$ is an injective homomorphism. Then $f(G_i)=hKh^{-1}$ where $K \subset H_j$ for some $j$.
\end{lemma}
\begin{proof}
    By Kurosh's theorem, $f(G_i)=F(X)*(*_{j\in J}h_jK_jh_j^{-1})$ where $J$ is some index set, $K_j \subset H_j$, $F(X)$ is the free group on $X\subset H$, and $h_j \in H$. Since $G_i$ is indecomposable, $f(G_i)$ is as well and so there is $j\in J$ such that $f(G_i)=h_jKh_j^{-1}$. 
\end{proof}
Next, we discuss an important commutative diagram that is defined with the help of Lemma 2.1. Assume prime decompositions $W_0\cong M_1\#...\# M_{k_0} \#l_0(S^1 \times S^2)$ and $W_1\cong N_1\#...\# N_{k_1} \#l_1(S^1 \times S^2)$. We write $\pi_1(W_0)=\pi_1(M_1)*...*\pi_1(M_{k_0})*F_{l_0}$ where each factor corresponds to the inclusion of $\pi_1(M_i)$ and $F_{l_0}$ corresponds to the inclusion of $\pi_1(l_0(S^1\times S^2))$. Similarly, we write $\pi_1(W_1)=\pi_1(N_1)*...*\pi_1(N_{k_1})*F_{l_1}$. We define $p_{M_i}:W_0\rightarrow M_i$ to be the map induced by collapsing all other summands. The induced map on fundamental group $(p_{M_i})_*$ is exactly the projection map $p_{\pi_1(M_i)}:\pi_1(M)\rightarrow \pi_1(M_i)$ onto the $\pi_1(M_i)$ factor. We define $p_{N_i}$ identically on $\pi_1(W_1)$. Let $h:\pi_1(W_0)\rightarrow \pi_1(W_1)$ be an injection. By Lemma 2.1, $h(\pi_1(M_i))=gKg^{-1}$ where $K$ is a subgroup of $\pi_1(N_j)$ for some $j$ and $g\in \pi_1(W_1)$ (if $g\in \pi_1(N_j)\subset \pi_1(W_1)$, we assume that $g=e$). Denote by $h_{M_i}:\pi_1(M_i)\rightarrow \pi_1(N_j)$ the injection such that $h|_{\pi_1(M_i)}=gh_{M_i}g^{-1}$. We call $h_{M_i}$ \textit{the} $M_i$\textit{-restriction of} $h$. Given that $M_i$ has index $i$, we simply write $h_i$ instead of $h_{M_i}$. We have the following commutative diagram.
\[\begin{tikzcd}[ampersand replacement=\&]
	{\pi_1(W_0)} \& {\pi_1(W_0)} \\
	{\pi_1(M_i)} \& {\pi_1(N_j)}
	\arrow["h", from=1-1, to=1-2]
	\arrow["{(p_{M_i})_*}"', from=1-1, to=2-1]
	\arrow["{(p_{N_j})_*}", from=1-2, to=2-2]
	\arrow["{h_i}", from=2-1, to=2-2]
\end{tikzcd}\]
Suppose that $h=\alpha=(\iota_1)_*^{-1}\circ (\iota_0)_*$, as in Section 1, and fix $i\in \{1,...,k_0\}$. Suppose that $(\iota_1)_*$ is an isomorphism on $\pi_1$. We have an injective map $\alpha_i:\pi_1(M_i)\rightarrow \pi_1(N_j)$ given by the $M_i$-restriction of $\alpha$. Given a topological space $L$, we let $j_{L}:L\rightarrow K(\pi_1(L),1)$ be the inclusion of $L$ into its $K(\pi,1)$ space. We have the following diagram, with arrows defined up to homotopy.
\[\begin{tikzcd}
	& {W_0} & W & {W_1} & \\
	& {K(\pi_1(W_0),1)} & {K(\pi_1(W),1)} & {K(\pi_1(W_1),1)} \\
	{M_i} & {K(\pi_1(M_i),1)} && {K(\pi_1(N_{j}),1)} & {N_{j}}
	\arrow["{\iota_0}", from=1-2, to=1-3]
	\arrow["{j_{W_0}}"', from=1-2, to=2-2]
	\arrow["{p_{M_i}}"', shift right=2, from=1-2, to=3-1]
	\arrow["{j_W}"', from=1-3, to=2-3]
	\arrow["{\iota_1}"', from=1-4, to=1-3]
	\arrow["{j_{W_1}}", from=1-4, to=2-4]
	\arrow["{p_{N_{j}}}", shift left=2, from=1-4, to=3-5]
	\arrow["{\iota_0}", from=2-2, to=2-3]
	\arrow["{q_{M_i}}"', from=2-2, to=3-2]
	\arrow["{\iota_1^{-1}}"', shift right, from=2-3, to=2-4]
	\arrow["{\iota_1}"', shift right, from=2-4, to=2-3]
	\arrow["{q_{N_{j}}}", from=2-4, to=3-4]
	\arrow["{j_{M_i}}", from=3-1, to=3-2]
	\arrow["f", from=3-2, to=3-4]
	\arrow["{j_{N_{j}}}"', from=3-5, to=3-4]
\end{tikzcd}\]
We refer to this diagram by $\text{D}*$. In D$*$, $q_{M_i}$ is defined uniquely up to homotopy so that $(q_{M_i})_*=(j_{M_i})_*\circ (p_{M_i})_*\circ (j_{W_0})_*^{-1}$. Similarly, $\iota_0:K(\pi_1(W_0),1)\rightarrow K(\pi_1(W),1)$ is defined so that $(\iota_0)_*=(j_W)_*\circ (\iota_0)_* \circ (j_{W_0})_*^{-1}$. The maps $q_{N_j}$ and $\iota_1:K(\pi_1(W_1),1)\rightarrow K(\pi_1(W),1)$ are defined analogously. The map $f$ is defined so that $f_*=(j_{N_j})_*\circ \alpha_i\circ (j_{M_i})_*^{-1}$. The top two rows and the left and right triangles commute up to homotopy, since they are constructed to commute on $\pi_1$, which completely determines the homotopy class of a map into a $K(\pi,1)$ space (see \cite{Hatcher2002AT}, for instance). The single square diagram defined above implies commutativity of the bottom row of D$*$. Indeed, the single square diagram gives that $(j_{N_j})_*^{-1}\circ f_*\circ (q_{M_i})_*\circ (j_{W_0})_*=(j_{N_j})_*^{-1}\circ (q_{N_j})_*\circ (j_{W_1})_* \circ (\iota_1)_*^{-1}\circ (\iota_0)_*$. Using commutativity of the top two rows and canceling $(j_{N_j})_*^{-1}$ on the left hand side, the previous equality becomes $f_*\circ (q_{M_i})_*\circ (j_{W_0})_*=(q_{N_j})_*\circ (\iota_1^{-1})_* \circ (j_W)_*\circ (j_W)^{-1}_* \circ (\iota_0)_*\circ (j_{W_0})_*$, which simplifies further to $f_*\circ (q_{M_i})_*=(q_{N_j})_*\circ (\iota_1^{-1})_* \circ (\iota_0)_*$. Thus, D$*$ commutes. We use this diagram in the following proof.
\begin{proposition}
    Suppose that $W_0 \leq W_1$ by a rational homology ribbon cobordism $W$. Assume prime decompositions $W_0\cong M_1\#...\# M_{k_0} \#l_0(S^1 \times S^2)$ and $W_1\cong N_1\#...\# N_{k_1} \#l_1(S^1 \times S^2)$. Assume that the inclusion $\iota_1$ induces an isomorphism on $\pi_1$. Then for each $1\leq i \leq k_0$ such that $M_i$ is aspherical there is $1 \leq j\leq k_1$ such that $\alpha(\pi_1(M_i))=g\pi_1(N_j)g^{-1}$ for some $g\in \pi_1(W_1)$. Furthermore, $M_i$ is oriented homeomorphic to $N_j$ by a map $h:M_i\rightarrow N_j$ such that $h_*=\alpha_i$.
\end{proposition}
\begin{proof}
Let $1\leq i\leq k_0$ be such that $M_i$ is aspherical. By Lemma 2.1, there is $1\leq j\leq k_1$ such that $\alpha(\pi_1(M_i))=gKg^{-1}$ for some subgroup $K \subset \pi_1(N_j)$ and $g \in \pi_1(W_1)$. Let $\alpha_i$ be the $M_i$-restriction of $\alpha$. We wish to show that $K=\pi_1(N_j)$. $N_j$ must be aspherical and so we may choose a map $f:M_i \rightarrow N_j$ such that $f_{*}$ is exactly $\alpha_i$. Since $M_i$ and $N_j$ are aspherical, the diagram D$*$ becomes
\[\begin{tikzcd}
	{W_0} & W & {W_1} \\
	{K(\pi_1(W_0),1)} & {K(\pi_1(W),1)} & {K(\pi_1(W_1),1)} \\
	{M_i} && {N_j}
	\arrow["{\iota_0}", from=1-1, to=1-2]
	\arrow["{j_{W_0}}"', from=1-1, to=2-1]
	\arrow["{p_{M_i}}"', shift right, curve={height=30pt}, from=1-1, to=3-1]
	\arrow["{j_W}"', from=1-2, to=2-2]
	\arrow["{\iota_1}"', from=1-3, to=1-2]
	\arrow["{j_{W_1}}", from=1-3, to=2-3]
	\arrow["{p_{N_j}}", shift left, curve={height=-30pt}, from=1-3, to=3-3]
	\arrow["{\iota_0}", from=2-1, to=2-2]
	\arrow["{q_{M_i}}"', from=2-1, to=3-1]
	\arrow["{\iota_1^{-1}}"', shift right, from=2-2, to=2-3]
	\arrow["{\iota_1}"', shift right, from=2-3, to=2-2]
	\arrow["{q_{N_j}}", from=2-3, to=3-3]
	\arrow["f", from=3-1, to=3-3]
\end{tikzcd}\]
By Proposition 2.2 of \cite{FriedlMisevZentner2025RibbonRH}, $(\iota_0)_*([W_0])=(\iota_1)_*([W_1])$. Since $(p_{M_i})_*([W_0])=[M_i]$ and $(p_{N_j})_*([W_1])=[N_j]$, the diagram above shows that $f_*([M_i])=[N_j]$. A degree $1$ map induces a surjection on $\pi_1$, and so the map $f$ induces an isomorphism on $\pi_1$. By Theorem 2.4 (2) of \cite{FriedlMisevZentner2025RibbonRH}, there is a homeomorphism $h:M_i \rightarrow N_j$ such that $h_*=f_*=\alpha_i$. Finally, since $N_j$ is aspherical, we have $h\simeq f$, implying that $h_*([M_i])=[N_j]$. Therefore, $h$ is an orientation preserving homeomorphism.
\end{proof}
Finally, we prove one last lemma, which will be used throughout the text. We note that an unspecified homology group refers to integer homology.
\begin{lemma}
    Suppose that $W$ is a rational homology ribbon cobordism where $W_0 \leq W_1$. Then $H_2(W,W_0)=0$. In particular, the inclusion induced map $H_1(W_0)\rightarrow H_1(W)$ is injective.
\end{lemma}
\begin{proof}
By Poincaré-Lefschetz, $H_2(W,W_0)\cong H^2(W,W_1)$. By Proposition 2.1 of \cite{DaemiLidmanVelaVickWong2022RibbonHomologyCobordisms}, $H_1(W,W_1)=0$. Hence, by the universal coefficient theorem for cohomology, we have that $H^2(W,W_1)\cong \text{Hom}(H_2(W,W_1),\mathbb{Z})$. By the universal coefficient theorem for homology, we have that $H_2(W,W_1;\mathbb{Q})\cong H_2(W,W_1)\otimes \mathbb{Q}$. Since $H_2(W,W_1;\mathbb{Q})=0$, we must have that $H_2(W,W_1)$ is pure torsion. Hence, $\text{Hom}(H_2(W,W_1),\mathbb{Z})=0$. We conclude that $H_2(W,W_0)=0$.
\end{proof}

\section{Reduction to the Case Without $S^1 \times S^2$ Summands}
In this section, we show that it suffices to prove Theorem 1.2 in the case that $W_0$ and $W_1$ have no $S^1 \times S^2$ summands in their prime decomposition. We start with definitions and conventions unique to this section.

Given a 3-manifold $M$, we denote by $M'$ a choice of 3-manifold with prime decomposition the irreducible summands of $M$. That is, if $M\cong M_1 \# ...\# M_k\# l(S^1 \times S^2)$ where each $M_i$ is irreducible, then $M'\cong M_1 \#...\#M_k$. The manifold $M$ therefore has the identification $M\cong M'\# F$ where $F\cong l(S^1 \times S^2)$. The prime decomposition of $M$ gives a decomposition of $\pi_1(M)$, given by $\pi_1(M)=G_1*...*G_k*F_l$ where $G_i \cong \pi_1(M_i)$ and $F_l \cong \pi_1(l(S^1 \times S^2))$. This provides a Grushko decomposition for $\pi_1(M)$. Such a decomposition is not unique, though Lemma 2.1 provides a stability result when applied to the identity map $\pi_1(M)\rightarrow \pi_1(M)$: given two decompositions $\pi_1(M)=G_1*...*G_k*F_{l,1}=H_1*...*H_k*F_{l,2}$, we must have that $G_i=gH_{\sigma(i)}g^{-1}$ for some $g\in \pi_1(M)$ and $\sigma\in S_k$. Hence, given any identification $M\cong M'\#F$, the inclusion $\pi_1(M') \rightarrow \pi_1(M)$ provides a canonical subspace of $H_1(M)$ given by the image of the inclusion $H_1(M')\rightarrow H_1(M)$. We denote this subspace of $H_1(M)$ by $H_1^{irr}(M)$. In rational homology, we write $H_1^{irr}(M;\mathbb{Q})$. We have the following result in the context of rational homology ribbon cobordisms between manifolds with the same number of $S^1 \times S^2$ summands.
\begin{lemma}
Let $W$ be a rational homology ribbon cobordism where $W_0 \leq W_1$. Assume that $\iota_1$ induces an isomorphism on $\pi_1$ and that $W_0$ and $W_1$ have the same number of $S^1 \times S^2$ summands in their prime decomposition. Then $(\iota_0)_*(H_1^{irr}(W_0;\mathbb{Q}))=(\iota_1)_*(H_1^{irr}(W_1;\mathbb{Q}))$.
\end{lemma}
\begin{proof}
By Lemma 2.1, the induced map $\overline{\alpha}=(\iota_1)^{-1}_*\circ(\iota_0)_*:H_1(W_0;\mathbb{Q})\rightarrow H_1(W_1;\mathbb{Q})$ satisfies $\overline{\alpha}(H_1^{irr}(W_0;\mathbb{Q}))\subset H_1^{irr}(W_1;\mathbb{Q})$. Since $W_0$ and $W_1$ have the same number of $S^1 \times S^2$ summands, by definition of $H_1^{irr}$, we must have that $rk(H_1^{irr}(W_0;\mathbb{Q}))=rk(H_1^{irr}(W_1;\mathbb{Q}))$, and so $\overline{\alpha}(H_1^{irr}(W_0;\mathbb{Q}))= H_1^{irr}(W_1;\mathbb{Q})$. We conclude that $(\iota_0)_*(H_1^{irr}(W_0;\mathbb{Q}))=(\iota_1)_*(H_1^{irr}(W_1;\mathbb{Q}))$
\end{proof}
\begin{remark}
Suppose that $W$ is a rational homology ribbon cobordism where $W_0\leq W_1$. By Theorem 1.5 of \cite{FriedlMisevZentner2025RibbonRH}, $\iota_1$ automatically induces an isomorphism on $\pi_1$ if $W_0\cong W_1$. Furthermore, if $W_0\cong W_1$, then, clearly, $W_0$ and $W_1$ have the same number of $S^1 \times S^2$ summands in their prime decompositions. In particular, the assumptions of Lemma 3.1 hold if $W_0 \cong W_1$.
\end{remark}
We use the case of Remark 3.2 applied to Lemma 3.1 to prove the following.
\begin{proposition}
Suppose that $W_0 \leq W_1$ and $W_1 \leq W_0$. Then $W_0$ and $W_1$ have the same number of $S^1 \times S^2$ summands in their prime decompositions.
\end{proposition}
\begin{proof}
Let $W$ be a rational homology ribbon cobordism such that $W_0 \leq W_1$ and let $V$ be a rational homology ribbon cobordism such that $W_1 \leq W_0$. We obtain a rational homology ribbon cobordism $WV$ from $W_0$ to $W_0$ by gluing both cobordisms along the $W_1$ boundary component. We have the following commutative diagram.
\[\begin{tikzcd}[ampersand replacement=\&]
	\&\& WV \&\& \\
	\\
	{W_0} \& W \& {W_1} \& V \& {W_0}
	\arrow["{\delta_{0}}"{pos=0.4}, from=3-1, to=1-3]
	\arrow["{\iota_0}"', from=3-1, to=3-2]
	\arrow["{\delta_W}"{pos=0.2}, from=3-2, to=1-3]
	\arrow["{\delta_{W_1}}"'{pos=0.2}, from=3-3, to=1-3]
	\arrow["{\iota_1}", from=3-3, to=3-2]
	\arrow["{\epsilon_0}"', from=3-3, to=3-4]
	\arrow["{\delta_V}"'{pos=0.2}, from=3-4, to=1-3]
	\arrow["{\delta_1}"'{pos=0.4}, from=3-5, to=1-3]
	\arrow["{\epsilon_1}", from=3-5, to=3-4]
\end{tikzcd}\]
where every map is an inclusion. By Corollary 1.6 of \cite{FriedlMisevZentner2025RibbonRH} and Proposition 1.2 of \cite{DaemiLidmanVelaVickWong2022RibbonHomologyCobordisms}, we may define the injective map $\alpha:\pi_1(W_0)\rightarrow \pi_1(W_1)$ given by $(\iota_1)_*^{-1}\circ(\iota_0)_*$. Similarly, we may define the injection $\beta:\pi_1(W_1)\rightarrow \pi_1(W_0)$ given by $(\epsilon_1)_*^{-1}\circ(\epsilon_0)_*$. By Remark 3.2 applied to Lemma 3.1, $(\delta _1)_*^{-1}\circ (\delta_0)_*$ sends $H_1^{irr}(W_0;\mathbb{Q})$ isomorphically onto $H_1^{irr}(W_0;\mathbb{Q})$. Since $\alpha$ and $\beta$ are injective, Lemma 2.1 implies that $(\iota_1)_*^{-1} \circ (\iota_0)_*(H_1^{irr}(W_0;\mathbb{Q}))\subset H_1^{irr}(W_1;\mathbb{Q})$ and $(\epsilon_1)_*^{-1} \circ (\epsilon_0)_*(H_1^{irr}(W_1;\mathbb{Q}))\subset H_1^{irr}(W_0;\mathbb{Q})$. Since $(\delta _1)_*^{-1}\circ (\delta_0)_*$ sends $H_1^{irr}(W_0;\mathbb{Q})$ isomorphically onto $H_1^{irr}(W_0;\mathbb{Q})$, the diagram shows that $(\epsilon_1)_*^{-1}\circ (\epsilon_0)_*$ sends $H_1^{irr}(W_1;\mathbb{Q})$ surjectively onto $H_1^{irr}(W_0;\mathbb{Q})$. Since $(\epsilon_1)_*^{-1}\circ (\epsilon_0)_*$ is injective on $H_1(W_1;\mathbb{Q})$, we have that $(\epsilon_1)_*^{-1}\circ (\epsilon_0)_*$ maps $H_1^{irr}(W_1;\mathbb{Q})$ isomorphically onto $H_1^{irr}(W_0;\mathbb{Q})$. Since $V$ is a rational homology cobordism, the map $(\epsilon_1)_*^{-1}\circ (\epsilon_0)_*$ is an isomorphism on all of $H_1(W_1;\mathbb{Q})$. In particular, we must have that the rank in first rational homology coming from non-irreducible prime summands in $W_1$ is equal to the rank in first rational homology coming from non-irreducible prime summands in $W_0$. That is, $W_0$ and $W_1$ have the same number of $S^1 \times S^2$ summands.
\end{proof}

Next, we demonstrate how $S^1\times S^2$ summands can be removed, and, in doing so, we provide a new rational homology ribbon cobordism, which we will denote by $W'$. Lemma 3.1 will assist in demonstrating that $W'$ is a rational homology ribbon cobordism between the boundary components $W_0'$ and $W_1'$ of $W'$.

Given a rational homology ribbon cobordism $W$ where $W_0 \leq W_1$ and such that $W_0$ and $W_1$ have the same number of $S^1 \times S^2$ summands, we produce a new ribbon cobordism $W'$. Let $W_0\cong W_0' \# F_0$ be an identification of $W_0$, where $F_0\cong l(S^1 \times S^2)$. Start with $W_0'\times [0,1]$ and attach to $W_0'\times \{1\}$ $l$ 1-handles so that the positive boundary component at this stage is exactly the provided identification $W_0' \# F_0$. Attach to this boundary component the negative boundary component of $W$ via the identification $W_0\cong W_0' \# F_0$. At this stage, the positive boundary component is $W_1$. The construction of $W'$ is completed by attaching $l$ 2-handles to an $S^1$ core of each $S^1 \times S^2$ summand of $W_1$ so that the resulting effect in the $W_1$ boundary component is the dual $0$-Dehn-surgery on the core of the $0$-Dehn-surgery solid torus. In doing so, the positive boundary becomes $W_1'$. We refer to Figure 1 for a demonstration of obtaining $W'$ from $W$.
\begin{figure}[h]
    \centering
    \includegraphics[width=0.8\textwidth]{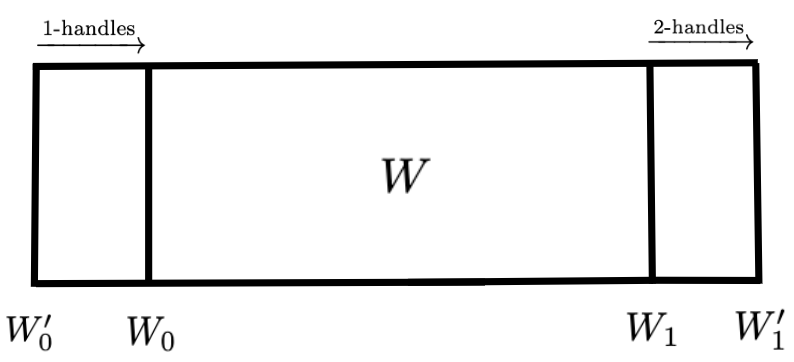}
    \caption{The ribbon cobordism $W'$.}
    \label{fig:ribboncob}
\end{figure}
Observe that $W_1'$ is obtained from $W_1$ by killing $S^1 \times S^2$ summands inside $W_1$. In particular, this gives an isomorphism $\pi_1(W_1)\cong\pi_1(W_1')*F_l$ and we define $p_{\pi_1(W_1')}:\pi_1(W_1)\rightarrow \pi_1(W_1')$ to be the composition $\pi_1(W_1)\xrightarrow{\cong}\pi_1(W_1')*F_l\xrightarrow{proj}\pi_1(W_1')$ where $proj$ denotes projection onto the $\pi_1(W_1')$ factor. We note that the map $p_{\pi_1(W_1')}$ has right inverse $i_{\pi_1(W_1')}$ defined by post-composing the canonical inclusion $\pi_1(W_1')\rightarrow \pi_1(W_1')*F_l$ with the isomorphism $\pi_1(W_1')*F_l\cong \pi_1(W_1)$. We also define $\iota_0'$ and $\iota_1'$ to be the inclusions of $W_0'$ and $W_1'$ into $W'$ and let $\iota_W$ denote the inclusion of $W$ into $W'$. Given our construction, we have the following commutative diagram, and we provide a short proof of commutativity.
\[\begin{tikzcd}[ampersand replacement=\&]
	{\pi_1(W_0')} \& {\pi_1(W')} \& {\pi_1(W_1')} \\
	{\pi_1(W_0)} \& {\pi_1(W)} \& {\pi_1(W_1)}
	\arrow["{(\iota_0')_*}", from=1-1, to=1-2]
	\arrow[from=1-1, to=2-1]
	\arrow["{(\iota_1')_*}"', from=1-3, to=1-2]
	\arrow["{i_{\pi_1(W_1')}}", shift left=3, from=1-3, to=2-3]
	\arrow["{(\iota_0)_*}"', from=2-1, to=2-2]
	\arrow["{(\iota_W)_*}", from=2-2, to=1-2]
	\arrow["{p_{\pi_1(W_1')}}", from=2-3, to=1-3]
	\arrow["{(\iota_1)_*}", from=2-3, to=2-2]
\end{tikzcd}\]
\begin{proof}
Let $P_0$ denote $W_0'\times [0,1]$ with $l$ 1-handles attached as above. Let $P_1$ denote $W_1'\times [0,1]$ with $l$ 2-handles attached to $W_1'\times \{0\}$ along $l$ unknots living in $l$ disjoint 3-balls so that the resulting surgery has the $0/1$ framing. Then $W'$ is constructed by attaching the positive boundary component of $P_0$ to the negative boundary component of $W$, and the negative boundary component of $P_1$ to the positive boundary component of $W$. This makes commutativity of the diagram above obvious by the following commutative diagram.
\[\begin{tikzcd}[ampersand replacement=\&]
	{\pi_1(W_0')} \& {\pi_1(P_0)} \& {\pi_1(W')} \& {\pi_1(P_1)} \& {\pi_1(W_1')} \\
	\& {\pi_1(W_0)} \& {\pi_1(W)} \& {\pi_1(W_1)}
	\arrow[from=1-1, to=1-2]
	\arrow["{i_{\pi_1(W_0')}}"', from=1-1, to=2-2]
	\arrow[from=1-2, to=1-3]
	\arrow[from=1-4, to=1-3]
	\arrow["\cong"', from=1-5, to=1-4]
	\arrow["{i_{\pi_1(W_1')}}", from=1-5, to=2-4]
	\arrow["\cong", from=2-2, to=1-2]
	\arrow[from=2-2, to=2-3]
	\arrow[from=2-3, to=1-3]
	\arrow[from=2-4, to=1-4]
	\arrow[from=2-4, to=2-3]
\end{tikzcd}\]
where all unlabeled arrows are induced by inclusions.
\end{proof}

Recall that, given a rational homology ribbon cobordism $W$ where $W_0 \leq W_1$ such that the inclusion $\iota_1$ of $W_1$ induces an isomorphism on $\pi_1$, we denote by $\alpha:\pi_1(W_0)\rightarrow \pi_1(W_1)$ the composition $(\iota_1)^{-1}_*\circ (\iota_0)_*$. Inspired by our construction above, we define $\alpha'$ as follows:
\[\begin{tikzcd}[ampersand replacement=\&]
	{\pi_1(W_0')} \& {\pi_1(W_0)} \& {\pi_1(W_1)} \& {\pi_1(W_1')}
	\arrow["{i_{\pi_1(W_0')}}", from=1-1, to=1-2]
	\arrow["\alpha", from=1-2, to=1-3]
	\arrow["{p_{\pi_1(W_1')}}", from=1-3, to=1-4]
\end{tikzcd}\]
We have the following result about $W'$.
\begin{proposition}
Let $W_0$ and $W_1$ be closed, connected, oriented 3-manifolds with an equal number of $S^1 \times S^2$ summands. Suppose that $W_0\leq W_1$ by a rational homology ribbon cobordism $W$ such that the inclusion $\iota_1$ induces an isomorphism on $\pi_1$. Then the cobordism $W'$ is a rational homology ribbon cobordism from $W_0'$ to $W_1'$ and $\iota_1'$ induces an isomorphism on $\pi_1$. Furthermore, $\alpha'=(\iota_1')_*^{-1}\circ (\iota_0')_*$.
\end{proposition}
\begin{proof}
The right square of the diagram above is given by
\[\begin{tikzcd}[ampersand replacement=\&]
	{\pi_1(W')} \& {\pi_1(W')} \\
	{\pi_1(W)} \& {\pi_1(W_1)}
	\arrow["{(\iota_1')_*}"', from=1-2, to=1-1]
	\arrow["{(\iota_W)_*}", from=2-1, to=1-1]
	\arrow["{p_{\pi_1(W')}}"', from=2-2, to=1-2]
	\arrow["{(\iota_1)_*}", from=2-2, to=2-1]
\end{tikzcd}\]
Since $(\iota_1)_*$ is an isomorphism and the additional 2-handles (those in $P_1$ rel $W_1$)  are attached along a generating set of $\ker(p_{\pi_1(W_1')})$, we have that $\ker((\iota_W)_*)=(\iota_1)_*(\ker(p_{\pi_1(W_1')}))$. Hence, we must have that $(\iota_1')_*$ is an isomorphism. Furthermore, the two square diagram also shows that $\alpha'=(\iota_1')_*^{-1}\circ (\iota_0')_*$. It remains to show that $W'$ is a rational homology cobordism. Again, using the diagram, Lemma 3.1, and the fact that $\iota_1'$ induces an isomorphism on $\pi_1$ we see that in rational homology $(\iota_0')_*$ is an isomorphism with image $(\iota_1')_*(H_1(W_1';\mathbb{Q}))$. By Poincaré-Lefschetz, we have that $H_j(W',W_i';\mathbb{Q})=0$ for $j=1,3$ and $i=0,1$. Since obtaining $W'$ requires attaching an equal number of 1-handles (resp. 3-handles) and 2-handles to $W_0'\times \{1\}$ (resp. $W_1'\times \{0\}$), the relative chain complex $C_{*}(W',W_i';\mathbb{Q})$ has Euler characteristic $0$. Therefore, we must have that $H_2(W',W_i';\mathbb{Q})=0$. Thus, $H_j(W',W_i';\mathbb{Q})=0$ for all $j$ and $i=0,1$.
\end{proof}
Combining Proposition 3.3 and Proposition 3.4, we obtain the following.
\begin{theorem}
Suppose that $W_0 \leq W_1$ and $W_1 \leq W_0$. Then $W_0'\leq W_1'$, $W_1'\leq W_0'$, and $W_0$ and $W_1$ have the same number of $S^1 \times S^2$ summands. In particular, if Theorem 1.2 is true on the set of all closed, connected, oriented 3-manifolds with no $S^1 \times S^2$ summands, then it is true for all closed, connected, oriented 3-manifolds.
\end{theorem}
\begin{proof}
Let $W$ be a rational homology cobordism where $W_0\leq W_1$. By Corollary 1.6 of \cite{FriedlMisevZentner2025RibbonRH}, $\iota_1$ induces an isomorphism on $\pi_1$. By Proposition 3.3, $W_0$ and $W_1$ have the same number of $S^1 \times S^2$ summands. In particular, the assumptions of Proposition 3.4 hold and so $W_0'\leq W_1'$. By symmetry, we also have that $W_1'\leq W_0'$. This proves the first part. 

The second part follows directly from the first: if Theorem 1.2 is true on the set of all closed, connected, oriented 3-manifolds with no $S^1 \times S^2$ summands, then by the first part, $W_0'\cong W_1'$ (where $\cong$ means oriented homeomorphic). In particular, $W_0\cong W_0' \# l(S^1 \times S^2)\cong W_1'\#l(S^1 \times S^2)\cong W_1$.
\end{proof}
\section{The Case Without $S^1 \times S^2$ Summands}
In this section, we show that Theorem 1.2 is true on the set of closed, connected, oriented 3-manifolds with no $S^1 \times S^2$ summands. First, we show that, if $W$ is such that $W_0\leq W_1$ and, furthermore, $W_1\leq W_0$, then $W$ is an integer homology cobordism. We prove the following lemma.
\begin{lemma}
Let $W$ be a rational homology cobordism where $W_0\leq W_1$ and $W_0\cong W_1$. Assume that $W_0$ and $W_1$ have no $S^1 \times S^2$ summands. Then $(\iota_0)_*:H_1(W_0)\rightarrow H_1(W)$ is an isomorphism.
\end{lemma}
\begin{proof}
By Lemma 2.3, the inclusion $\iota_0$ induces an injection on $H_1(\_,\mathbb{Z})$. Furthermore, by Theorem 1.5 of \cite{FriedlMisevZentner2025RibbonRH}, $\iota_1$ induces an isomorphism on $\pi_1$, and hence $H_1(\_,\mathbb{Z})$. Therefore, the map $\alpha^{ab}:\bigoplus_{i=1}^{k}H_1(M_i)\rightarrow \bigoplus_{i=1}^{k}H_1(N_i)$, which is equal to $(\iota_1)_*^{-1}\circ(\iota_0)_*:H_1(W_0)\rightarrow H_1(W_1)$, is an injection. By Lemma 2.1, we  know that, for each $i$, $\alpha^{ab}(H_1(M_i))\subset H_1(N_j)$ for some $j$. By Proposition 2.2, if $M_i$ is an aspherical summand, we must have that $\alpha^{ab}(H_1(M_i))=H_1(N_j)$. Since $W_0\cong W_1$, $\text{rk}(H_1(W_0))=\text{rk}(H_1(W_1))$. Furthermore, since $W$ is a rational homology cobordism, the universal coefficient theorem for homology gives that $H_1(W_0,\mathbb{Q})\cong H_1(W_0)\otimes \mathbb{Q}$, $H_1(W_1,\mathbb{Q})\cong H_1(W_1)\otimes \mathbb{Q}$, and that the induced map $(\iota_1)_*^{-1}\circ (\iota_0)_*$ on rational homology is given by $(\alpha^{ab}\otimes \mathbb{Q})(x\otimes q)=\alpha^{ab}(x)\otimes q$. Hence, the map $\alpha^{ab}\otimes \mathbb{Q}$ defines a map $\alpha^{ab}\otimes\mathbb{Q}:\bigoplus_{i=1}^{k}(H_i(M_i)\otimes \mathbb{Q})\rightarrow \bigoplus_{i=1}^{k}(H_i(N_i)\otimes \mathbb{Q})$. Since $\alpha^{ab}\otimes\mathbb{Q}$ is an isomorphism and each summand maps stably into another summand, any summand with non-zero free rank in integral homology must lie in the image of $\alpha^{ab}$. That is, if $j$ is such that $\text{rk}(H_1(N_j))>0$, then there is $i$ such that $\alpha^{ab}(H_1(M_i))\subset H_1(N_j)$. Since only aspherical summands can have $\text{rk}(H_1(N_j))>0$, the previous inclusion is in fact an equality. Hence, $\alpha^{ab}$ may be viewed as a map $\alpha^{ab}:\mathbb{Z}^b\oplus G\rightarrow \mathbb{Z}^b\oplus G$ where $G$ is a finite abelian group and such that $\alpha^{ab}(\mathbb{Z}^b)=\mathbb{Z}^b$. Since $\alpha^{ab}$ is injective, we must have that $\alpha^{ab}(G)=G$. This shows that $\alpha^{ab}$ is an isomorphism. Therefore, the inclusion $\iota_0$ induces an isomorphism on homology.
\end{proof}
 Using the result above, we have the following.
\begin{theorem}
Suppose that $W_0\leq W_1$ and $W_1 \leq W_0$. Assume that $W_0$ and $W_1$ have no $S^1 \times S^2$ summands. Let $W$ be a rational homology ribbon cobordism where $W_0\leq W_1$. Then $W$ is an integer homology cobordism.
\end{theorem}
\begin{proof}
We glue cobordisms, as in the proof of Proposition 3.3. Let $V$ be a rational homology ribbon cobordism such that $W_1 \leq W_0$. We obtain a rational homology ribbon cobordism $WV$ from $W_0$ to $W_0$ by gluing both cobordisms along the $W_1$ boundary component. As in Proposition 3.3, we have the following commutative diagram.
\[\begin{tikzcd}[ampersand replacement=\&]
	\&\& WV \&\& \\
	\\
	{W_0} \& W \& {W_1} \& V \& {W_0}
	\arrow["{\delta_{0}}"{pos=0.4}, from=3-1, to=1-3]
	\arrow["{\iota_0}"', from=3-1, to=3-2]
	\arrow["{\delta_W}"{pos=0.2}, from=3-2, to=1-3]
	\arrow["{\delta_{W_1}}"'{pos=0.2}, from=3-3, to=1-3]
	\arrow["{\iota_1}", from=3-3, to=3-2]
	\arrow["{\epsilon_0}"', from=3-3, to=3-4]
	\arrow["{\delta_V}"'{pos=0.2}, from=3-4, to=1-3]
	\arrow["{\delta_1}"'{pos=0.4}, from=3-5, to=1-3]
	\arrow["{\epsilon_1}", from=3-5, to=3-4]
\end{tikzcd}\]
Here, every map is an inclusion. By Lemma 4.1, $(\delta_1)^{-1}_*\circ (\delta_0)_*:H_1(W_0)\rightarrow H_1(W_0)$ is an isomorphism. Since, on $H_1$, $(\delta _1)_*^{-1}\circ (\delta_0)_*=(\epsilon_1)_*^{-1}\circ(\epsilon_0)_*\circ(\iota_1)_*^{-1} \circ(\iota_0)_*$, we obtain that $(\epsilon_1)_*^{-1}\circ(\epsilon_0)_*$ is surjective and hence, by Lemma 2.3, an isomorphism. From this, also using Lemma 2.3, we obtain that $(\iota_1)_*^{-1} \circ(\iota_0)_*$ is an isomorphism. In particular, $(\iota_0)_*$ induces an isomorphism on $H_1$. Hence, we have that $H_1(W,W_0)=0$. By Lemma 2.3, $H_2(W,W_0)=0$. Furthermore, $H_3(W,W_0)\cong H^1(W,W_1)\cong\text{Hom}(H_1(W,W_1),\mathbb{Z})\cong0$. Thus, $H_i(W,W_0)=0$ for all $i$. By Poincaré-Lefschetz, $H_i(W,W_1)=0$ for all $i$. Hence, $W$ is an integer homology cobordism. 
\end{proof}
An integer homology cobordism alone is already quite restrictive. Now, we wish to show that $\alpha:\pi_1(W_0)\rightarrow \pi_1(W_1)$ is `almost an isomorphism'. By this, we mean the following: if $\pi_1(W_0)=\pi_1(M_1)*...*\pi_1(M_{k_0})$ and $\pi_1(W_1)=\pi_1(N_1)*...*\pi_1(N_{k_1})$ are Grushko decompositions, then $k_0=k_1$ and there is $\sigma\in S_{k_0}$ such that $\alpha(\pi_1(M_i))=g_i\pi_1(N_{\sigma(i)})g_i^{-1}$ for some $g_i\in \pi_1(W_1)$. To do this, we work in $K(\pi,1)$ spaces. 

A 3-manifold $M=M_1\#...\#M_k$ has $K(\pi,1)$ constructed by attaching 3-balls to each attaching sphere of the prime decomposition and then attaching the necessary cells in each summand so that all homotopy groups vanish. As mentioned in Section 2, we denote by $j_M$ the canonical inclusion $j_M:M\rightarrow K(\pi_1(M),1)$. This map induces an isomorphism on $\pi_1$, and sends the generator $[M]$ to $(j_{M_1})_*([M_1])+...+(j_{M_k})_*([M_k])\in H_3(K(\pi_1(M),1))\cong \bigoplus_{i=1}^{k} H_3(K(\pi_1(M_i),1))$. If $M_i$ is aspherical, then $(j_{M_i})_*$ is simply the identity on $H_3$. If $M_i$ is spherical, then, by Theorem 9.1 of \cite{brown1982cohomology}, $H_3(K(\pi_1(M_i),1))\cong \mathbb{Z}/|\pi_1(M_i)|\mathbb{\cdot Z}$ and the map $(j_{M_i})_*$ sends $[M_i]$ to a generator of $H_3(K(\pi_1(M_i),1))$. Given the isomorphism $H_3(K(\pi_1(M),1))\cong \bigoplus_{i=1}^{k} H_3(K(\pi_1(M_i),1))$, we have a projection map $(p_{M_i})_*:H_3(K(\pi_1(M),1)) \rightarrow H_3(K(\pi_1(M_i),1))$. Finally, we use the following convention: given a map $f:\pi_1(M)\rightarrow \pi_1(N)$, we denote also by $f$ the induced map $K(\pi_1(M),1)\xrightarrow{f}K(\pi_1(N),1)$, which is defined uniquely up to homotopy. We prove the following lemma.
\begin{lemma}
Let $M=M_1\#...\#M_k$ be a closed, connected, oriented 3-manifold with prime decomposition consisting of only irreducible summands. Let $f:\pi_1(M)\rightarrow \pi_1(M)$ be an injective endomorphism of $\pi_1(M)$. Let $i\in \{1,...,k\}$ and suppose that $f(\pi_1(M_j))$ is not contained in a conjugate of $\pi_1(M_i)$ for all $j\in \{1,...,k\}$. Then $(p_{\pi_1(M_i)})_*\circ f_*:H_n(K(\pi_1(M),1))\rightarrow H_n(K(\pi_1(M),1))$ is the zero map.
\end{lemma}
\begin{proof}
Let $f:K(\pi_1(M),1)\rightarrow K(\pi_1(M),1)$ be the induced map on $K(\pi,1)$ spaces. Recall that $K(\pi_1(M),1)$ may be decomposed as the wedge of $K(\pi,1)$ spaces of the summands. Hence, if we let $f_j:K(\pi_1(M_j),1)\rightarrow K(\pi_1(M),1)$ be the restriction map $f|_{K(\pi_1(M_j),1)}$, then we have that $f_*:H_n(K(\pi_1(M),1))\rightarrow H_n(K(\pi_1(M),1))$ decomposes as $\bigoplus_{j=1}^{k}(f_j)_*$. On fundamental groups, by Lemma 2.1, the map $f_j$ has image $g_jK_jg_j^{-1}$ where $g_j\in \pi_1(M)$ and $K_j\subset \pi_1(M_{b_j})$ for some $b_j\in \{1,...,k\}$. Let $h_j:\pi_1(M_j)\rightarrow \pi_1(M_{b_j})$ be the map such that $f_j=g_{j}h_jg_j^{-1}$. Denote also by $h_j$ the induced map on $K(\pi,1)$ spaces and by $\delta_{b_j}$ the natural inclusion $K(\pi_1(M_{b_j}),1)\hookrightarrow K(\pi_1(M),1)$. The maps $\delta_{b_j}\circ h_{j}$ and $f_j$ induce the same maps on $\pi_1$ up to conjugation, and are hence homotopy equivalent (see \cite{Hatcher2002AT}, for instance). In particular, the map $f_j$ satisfies $f_j(H_n(K(\pi_1(M_j),1)))\subset H_n(K(\pi_1(M_{b_j}),1))$. Since $f=\bigoplus_{j=1}^{k}(f_j)_*$ and $b_j\neq i$ for all $j$, we must have that $(p_{\pi_1(M_i)})_*\circ f_*=0$.
\end{proof}
Using this lemma, we may prove the following.
\begin{lemma}
Suppose that $W_0\leq W_1$. Let $W$ be a rational homology ribbon cobordism from $W_0$ to $W_1$. Suppose that $W_0\cong W_1\cong M_1\#...\#M_{k}$ where each $M_i$ is irreducible. Then there is a permutation $\sigma\in S_{k}$ such that $\alpha(\pi_1(M_i))=g_i\pi_1(M_{\sigma(i)})g_i^{-1}$.
\end{lemma}
\begin{proof}
Recall the top two rows of the diagram D$*$, which commute up to homotopy.
\[\begin{tikzcd}
	{W_0} & W & {W_1} \\
	{K(\pi_1(W_0),1)} & {K(\pi_1(W),1)} & {K(\pi_1(W_1),1)}
	\arrow["{\iota_0}", from=1-1, to=1-2]
	\arrow["{j_{W_0}}"', from=1-1, to=2-1]
	\arrow["{j_W}"', from=1-2, to=2-2]
	\arrow["{\iota_1}"', from=1-3, to=1-2]
	\arrow["{j_{W_1}}", from=1-3, to=2-3]
	\arrow["{\iota_0}", from=2-1, to=2-2]
	\arrow["{\iota_1^{-1}}"', shift right, from=2-2, to=2-3]
	\arrow["{\iota_1}"', shift right, from=2-3, to=2-2]
\end{tikzcd}\]
By Proposition 2.2 of \cite{FriedlMisevZentner2025RibbonRH}, $(\iota_0)_*([W_0])=(\iota_1)_*([W_1])$. Hence, the diagram above shows that $\alpha_*((j_{M_1})_*([M_1])+...+(j_{M_k})_*([M_k]))=(j_{M_1})_*([M_1])+...+(j_{M_k})_*([M_k])$ in $H_3$ (recall our previous discussion on $K(\pi,1)$ spaces). In particular, $(p_{\pi_1(M_i)})_*\circ \alpha_*:H_3(K(\pi_1(W_0),1))\rightarrow H_3(K(\pi_1(M_i),1))$ is not the zero map for any $i\in \{1,...,k\}$. By Lemma 4.3, for any $i\in\{1,...,k\}$, there must be $j\in \{1,...,k\}$ such that $\alpha(\pi_1(M_j))$ is contained in a conjugate of $\pi_1(M_i)$. That is, there is $\sigma\in S_{k}$ such that $\alpha(\pi_1(M_i))\subset g_iK_ig_i^{-1}$ where $g_i\in \pi_1(W_1)$ and $K_i\subset \pi_1(M_{\sigma(i)})$. 
By Proposition 2.2, if $i$ is such that $M_i$ is aspherical, then $K_i=\pi_1(M_{\sigma(i)})$. Since $\sigma$ is a permutation, there is $n$ such that $\alpha^n(\pi_1(M_i))$ is contained in a conjugate of itself (e.g. use $n$ the order of $\sigma$). Define $\alpha_{i,m}:\pi_1(M_{\sigma^{m}(i)})\rightarrow \pi_1(M_{\sigma^{m+1}(i)})$ to be the map such that $\alpha|_{\pi_1(M_{\sigma^m(i)})}=g_{\sigma^m(i)}\alpha_{i,m}g_{\sigma^m(i)}^{-1}$. Then $\alpha_{i,n-1}\circ ...\circ \alpha_{i,1}\circ \alpha_{i,0}$ is an injective map from $\pi_1(M_i)$ to $\pi_1(M_i)$. Hence, if $M_i$ is spherical, $\alpha_{i,n-1}\circ ...\circ \alpha_{i,1}\circ \alpha_{i,0}$ is an isomorphism. We then have that $\alpha_{i,n-1}$ is surjective. Since this map is injective as well, it is an isomorphism. By induction, $\alpha_{i,0}$ is an isomorphism - that is, when $M_i$ is spherical, $K_i=\pi_1(M_{\sigma(i)})$. Since all summands are irreducible, $\alpha(\pi_1(M_i))=g_i\pi_1(M_{\sigma(i)})g_{i}^{-1}$ for all $i\in \{1,...,k\}$.
\end{proof}
Lemma 4.4 allows us to prove the general case: when $W_0\leq W_1$ and $W_1 \leq W_0$ but $W_0$ and $W_1$ are not necessarily homeomorphic.
\begin{proposition}
Suppose that $W_0\leq W_1$ and $W_1\leq W_0$. Let $W$ be a rational homology ribbon cobordism from $W_0$ to $W_1$. Suppose that $W_0\cong M_1\#...\#M_{k_0}$ and $W_1\cong N_1\#...\#N_{k_1}$ where each $M_i$ and $N_i$ are irreducible. Then $k_0=k_1$ and there is a permutation $\sigma\in S_{k_0}$ such that $\alpha(\pi_1(M_i))=g_i\pi_1(N_{\sigma(i)})g_i^{-1}$.
\end{proposition}
\begin{proof}
Let $V$ be a rational homology ribbon cobordism where $W_1\leq W_0$. The cobordism $WV$ as constructed in Theorem 4.2 is a rational homology ribbon cobordism from $W_0$ to $W_0$. We label inclusions following the diagram in Theorem 4.2. Let $\beta:\pi_1(W_1)\rightarrow \pi_1(W_0)$ be the induced injective map on fundamental groups by $V$; that is, $\beta=(\epsilon_1)^{-1}_* \circ (\epsilon_0)_*$. Then $WV$ has induced map on fundamental groups $(\delta_1)_*^{-1}\circ (\delta_0)_*$ which is also given by $\beta\circ\alpha$. By Lemma 4.4, there is $\sigma\in S_{k_0}$ such that $\beta\circ\alpha(\pi_1(M_i))=g_{i,0}\pi_1(M_{\sigma(i)})g^{-1}_{i,0}$, where $g_{i,0}\in \pi_1(W_0)$. Since both $\alpha$ and $\beta$ are injective, this can happen only if $k_0\leq k_1$ and $\alpha(\pi_1(M_i))=g_{i,1}\pi_1(N_{j_i})g_{i,1}^{-1}$ for some $g_{i,1}\in \pi_1(W_1)$ and $j_i$ satisfying $j_i\neq j_b$ for $i\neq b$. Applying the same argument to the rational homology ribbon cobordism $VW$, we must also have that $k_1\leq k_0$, which concludes the proof.
\end{proof}
We require one more ingredient for proving Theorem 1.2. Suppose that $W_0 \leq W_1$ and $W_1\leq W_0$. In terms of proving Theorem 1.2, Theorem 3.5 shows that it suffices to assume that $W_0$ and $W_1$ have prime decomposition consisting of only irreducible summands. Suppose that $W$ is a rational homology ribbon cobordism where $W_0 \leq W_1$. Proposition 4.5 gives control over the fundamental group of irreducible summands, which, as we will see, deals with both aspherical summands and spherical summands that are not a lens space. For lens spaces, we use Theorem 1 of \cite{Eismeier2024Fourier}. We state this theorem below, as Theorem 4.6.
\begin{theorem}
Suppose that $L$ and $L'$ are connected sums of lens spaces. If $L$ and $L'$ are integer homology cobordant by a cobordism $W$, then $L$ is oriented diffeomorphic to $L'$, and the induced map $W_*:H_1(L)\rightarrow H_1(L')$ respects the natural direct sum decompositions. Furthermore, if $Y$ is any closed, oriented 3-manifold and $L\#Y$ is integer homology cobordant to $L'\#Y$, then $L$ is oriented diffeomorphic to $L'$.
\end{theorem}
In particular, in the following theorem, we show that each summand that is not a lens space is oriented homeomorphic to its corresponding summand under $\alpha$ given by Proposition 4.5, then we use Theorem 4.6 to finish.
\begin{theorem}
Let $W_0$ and $W_1$ be closed, connected, oriented 3-manifolds with no $S^1 \times S^2$ summands. Suppose that $W_0 \leq W_1$ and $W_1\leq W_0$. Then $W_0$ and $W_1$ are oriented homeomorphic.
\end{theorem}
\begin{proof}
Write prime decompositions $W_0\cong M_1\#...\#M_{k_0}$ and $W_1\cong N_1\#...\#N_{k_1}$. Let $W$ be a rational homology ribbon cobordism from $W_0$ to $W_1$. By Proposition 4.5, $k_0=k_1$ and the map $\alpha:\pi_1(W_0)\rightarrow \pi_1(W_1)$ is such that there is a permutation $\sigma\in S_{k_0}$ for which $\alpha(\pi_1(M_i))=g_i\pi_1(N_{\sigma(i)})g_i^{-1}$ where $g_i \in \pi_1(W_1)$. By Proposition 2.2, $M_i\cong N_{\sigma(i)}$ if $M_i$ is aspherical. If $M_i$ is spherical and not a lens space, then $M_i\cong \pm N_{\sigma(i)}$, since such manifolds are determined by their fundamental group up to orientation. To see that they must be oriented homeomorphic, we consider the diagram D$*$ from Section 2, given by
\[\begin{tikzcd}
	& {W_0} & W & {W_1} & \\
	& {K(\pi_1(W_0),1)} & {K(\pi_1(W),1)} & {K(\pi_1(W_1),1)} \\
	{M_i} & {K(\pi_1(M_i),1)} && {K(\pi_1(N_{\sigma(i)}),1)} & {N_{\sigma(i)}}
	\arrow["{\iota_0}", from=1-2, to=1-3]
	\arrow["{j_{W_0}}"', from=1-2, to=2-2]
	\arrow["{p_{M_i}}"', shift right=2, from=1-2, to=3-1]
	\arrow["{j_W}"', from=1-3, to=2-3]
	\arrow["{\iota_1}"', from=1-4, to=1-3]
	\arrow["{j_{W_1}}", from=1-4, to=2-4]
	\arrow["{p_{N_{\sigma(i)}}}", shift left=2, from=1-4, to=3-5]
	\arrow["{\iota_0}", from=2-2, to=2-3]
	\arrow["{q_{M_i}}"', from=2-2, to=3-2]
	\arrow["{\iota_1^{-1}}"', shift right, from=2-3, to=2-4]
	\arrow["{\iota_1}"', shift right, from=2-4, to=2-3]
	\arrow["{q_{N_{\sigma(i)}}}", from=2-4, to=3-4]
	\arrow["{j_{M_i}}", from=3-1, to=3-2]
	\arrow["{f}", from=3-2, to=3-4]
	\arrow["{j_{N_{\sigma(i)}}}"', from=3-5, to=3-4]
\end{tikzcd}\]
Since the construction of $K(\pi,1)$ of a spherical 3-manifold requires cells of dimension only 4 and higher, the map $f\circ j_{M_i}$ may be homotoped so that $(f\circ j_{M_i})(M_i)\subset j_{N_{\sigma(i)}}(N_{\sigma(i)})$. Denote this map $f\circ j_{M_i}$ with codomain restricted to $N_{\sigma(i)}$ by $f_r$. Commutativity of the diagram in $H_3$ and the relation $j_{N_{\sigma(i)}}\circ f_r=f\circ j_{M_i}$ shows that $f_r([M_i])=\pm[N_{\sigma(i)}]$. If $f_r([M_i])=-[N_{\sigma(i)}]$, then the relation $j_{N_{\sigma(i)}}\circ f_r=f\circ j_{M_i}$ implies that $f_*((j_{M_i})_*([M_i]))=-(j_{N_{\sigma(i)}})_*([N_{\sigma(i)}])$ in $H_3$. But Proposition 2.2 of \cite{FriedlMisevZentner2025RibbonRH} and the diagram above show that $f_*((j_{M_i})_*([M_i]))=(j_{N_{\sigma(i)}})_*([N_{\sigma(i)}])$ and since $(j_{N_{\sigma(i)}})_*([N_{\sigma(i)}])\neq -(j_{N_{\sigma(i)}})_*([N_{\sigma(i)}])$ (recall our discussion of $K(\pi,1)$ spaces after the proof of Theorem 4.2), we must have that $(f_r)_*([M_i])=[N_{\sigma(i)}]$. Hence, the map $f_r$ is a degree 1 homotopy equivalence. Section 1.3 of \cite{mccullough2000isometries} states that any oriented spherical 3-manifold which is not a lens space may only admit a degree 1 homotopy equivalence to itself, showing that $M_i\cong N_{\sigma(i)}$. 

At this point, we have shown that $M_i \cong N_{\sigma(i)}$ whenever $M_i$ is not a lens space. Furthermore, by Proposition 4.5, $\pi_1(M_i)\cong \pi_1(N_{\sigma(i)})$ for all $i$, and so whenever $M_i$ is a lens space, $N_{\sigma(i)}$ is as well. Thus, there is a connected oriented 3-manifold $Y$ such that $W_0 \cong Y\#L$ and $W_1 \cong Y \#L'$ where $L$ and $L'$ are connected sums of lens spaces. By Theorem 4.6, $L$ and $L'$ are oriented homeomorphic. We obtain that $W_0$ and $W_1$ are oriented homeomorphic.
\end{proof}
Combining Theorem 3.5 and Theorem 4.7 yields Theorem 1.2.
\begin{corollary}[Theorem 1.2]
Suppose that $W_0 \leq W_1$ and $W_1 \leq W_0$. Then $W_0$ and $W_1$ are oriented homeomorphic.
\end{corollary}
\bibliographystyle{alpha}
\bibliography{references2}

\end{document}